\documentclass[review]{elsarticle}
\usepackage{amsfonts,epsf,amsmath,amssymb,graphicx}
\usepackage{epstopdf}
\usepackage[top=2.5cm,bottom=2.5cm,left=2.5cm,right=2.5cm]{geometry}
\usepackage{lineno,hyperref}
\usepackage[algoruled,linesnumbered]{algorithm2e}
\usepackage{algorithmic}
\modulolinenumbers[5]

\newtheorem{theorem}{\bf Theorem}[section]

\newtheorem{lemma}[theorem]{\bf Lemma}

\newtheorem{definition}[theorem]{\bf Definition}

\journal{.}

\begin{document}

\begin{frontmatter}

\title{A Method for Computing the Edge-Hyper-Wiener Index of Partial Cubes and an Algorithm for Benzenoid Systems}

\author{Niko Tratnik}
\ead{niko.tratnik1@um.si}
\address{Faculty of Natural Sciences and Mathematics, University of Maribor, Slovenia}

%
%
%
%
%

\begin{abstract}
The edge-hyper-Wiener index of a connected graph $G$ is defined as $WW_e(G) = \frac{1}{2}\sum_{\lbrace e,f\rbrace \subseteq E(G)}d(e,f) + \frac{1}{2}\sum_{\lbrace e,f\rbrace \subseteq E(G)}d(e,f)^2$. We develop a method for computing the edge-hyper-Wiener index of partial cubes, which constitute a large class of graphs with a lot of applications. It is also shown how the method can be applied to trees. Furthermore, an algorithm for computing the edge-hyper-Wiener index of benzenoid systems is obtained. Finally, the algorithm is used to correct already known closed formulas for the edge-Wiener index and the edge-hyper-Wiener index of linear polyacenes.
\end{abstract}


\end{frontmatter}

\linenumbers

\section{Introduction}

The hyper-Wiener index is a distance-based graph invariant, used as a structure-descriptor for predicting physico–chemical properties of organic compounds (often those significant for pharmacology, agriculture, environment-protection etc.). It was introduced in 1993 by M. Randi\' c \cite{randic} and has been extensively studied in many papers. Randi\' c's original definition of the hyper-Wiener index was applicable just to trees and therefore, the hyper-Wiener index was later defined for all graphs \cite{klein}. The hyper-Wiener index is closely related to the well known Wiener index \cite{Wiener}, which is one of the most popular molecular descriptors. Moreover, it is also connected to the Hosoya polynomial \cite{hosoya}, since the relationship between them was proved in \cite{cash}.

In \cite{sklavzar-2000} a method for computing the hyper-Wiener index of partial cubes was developed and later it was applied to some chemical graphs \cite{cash-2002,klavzar-2000,zigert-2000}. Partial cubes constitute a large class of graphs with a lot of applications and includes, for example, many families of chemical graphs (benzenoid systems, trees, phenylenes, cyclic phenylenes, polyphenylenes).

The Wiener index and the hyper-Wiener index are based on the distances between pairs of vertices in a graph and therefore, similar concepts have been introduced for distances between pairs of edges under the names the edge-Wiener index \cite{iran-2009} and the edge-hyper-Wiener index \cite{edge-hyper}, respectively. With other words, the edge-hyper-Wiener index of a graph $G$ is just the hyper-Wiener index of the line graph $L(G)$ and a similar result holds for the edge-Wiener index. For some recent studies on the edge-hyper-Wiener index see \cite{tratnik,azari,soltani-2,soltani}.

In this paper we develop a method for computing the edge-hyper-Wiener index of partial cubes. Our main result is parallel to the result from \cite{sklavzar-2000}. However, the present method is a bit more difficult and the proof requires some additional insights. We use the cut method to reduce the problem of calculating the edge-hyper-Wiener index of partial cubes to the problem of calculating the edge-Wiener index and the contributions of pairs of $\Theta$-classes (for more information about the cut method see \cite{klavzar-arani}). As an example, we describe how the method can be used on trees. In the case of polycyclic molecules the calculation of the edge-hyper-Wiener index by the definition is not easy, especially if one is
interested in finding general expressions of homologous series. Therefore, using the main result and the results from \cite{kelenc}, we develop an algorithm for computing the edge-hyper-Wiener index of benzenoid systems, which is much more efficient than the calculation by the definition. Finally, our algorithm is used to obtain closed formulas for the edge-Wiener index and the edge-hyper-Wiener index of linear polyacenes. Although these formulas were previously developed in \cite{khormali,edge-hyper}, the results are not correct. 

\section{Preliminaries}

Unless stated otherwise, the graphs considered in this paper are connected. We define $d(x,y)$ to be the usual shortest-path distance between vertices $x, y \in V(G)$. 
\bigskip

\noindent
The \textit{Wiener index} and the \textit{edge-Wiener index} of a connected graph $G$ are defined  in the following way:
$$W(G) = \sum_{\lbrace u,v\rbrace \subseteq V(G)}d(u,v), \ \ \quad W_e(G) = \sum_{\lbrace e,f\rbrace \subseteq E(G)}d(e,f).$$

\noindent
The distance $d(e,f)$ between edges $e$ and $f$ of graph $G$ is defined as the distance between vertices $e$ and $f$ in the line graph $L(G)$. Here we follow this convention because in this way the pair $(E(G),d)$ forms a metric space. On the other hand, for edges $e = ab$ and $f = xy$ of a graph $G$ it is also legitimate to set
\begin{equation*}
\label{eq:hat-d}
\widehat{d}(e,f) = \min \lbrace d(a,x), d(a,y), d(b,x), d(b,y) \rbrace.
\end{equation*}
Replacing $d$ with $\widehat{d}$, we obtain another variant of the edge-Wiener index:
$$\widehat{W}_e(G) = \sum_{\lbrace e,f\rbrace \subseteq E(G)}\widehat{d}(e,f).$$

\noindent
Obviously, if $e$ and $f$ are two different edges of $G$, then $d(e,f)= \widehat{d}(e,f)+1$. Therefore, for any $G$ it holds
\begin{equation} \label{zveza_wie} W_e(G) = \widehat{W}_e(G) + \binom{|E(G)|}{2}.
\end{equation}

\noindent
The \textit{hyper-Wiener index} and the \textit{edge-hyper-Wiener index} of $G$ are defined as:
$$WW(G) = \frac{1}{2}\sum_{\lbrace u,v\rbrace \subseteq V(G)}d(u,v) + \frac{1}{2}\sum_{\lbrace u,v\rbrace \subseteq V(G)}d(u,v)^2,$$
$$WW_e(G) = \frac{1}{2}\sum_{\lbrace e,f\rbrace \subseteq E(G)}d(e,f) + \frac{1}{2}\sum_{\lbrace e,f\rbrace \subseteq E(G)}d(e,f)^2.$$

\noindent
It is easy to see that for any connected graph $G$ it holds $W_e(G) = W(L(G))$ and $WW_e(G) = WW(L(G))$.
\bigskip

\noindent
The {\em hypercube} $Q_n$ of dimension $n$ is defined in the following way: 
all vertices of $Q_n$ are presented as $n$-tuples $(x_1,x_2,\ldots,x_n)$ where $x_i \in \{0,1\}$ for each $1\leq i\leq n$ 
and two vertices of $Q_n$ are adjacent if the corresponding $n$-tuples differ in precisely one coordinate. Therefore, the \textit{Hamming distance} between two tuples $x$ and $y$ is the number of positions in $x$ and $y$ in which they differ.
\bigskip

\noindent
A subgraph $H$ of a graph $G$ is called an \textit{isometric subgraph} if for each $u,v \in V(H)$ it holds $d_H(u,v) = d_G(u,v)$. Any isometric subgraph of a hypercube is called a {\em partial cube}. For an edge $ab$ of a graph $G$, let $W_{ab}$ be the set of vertices of $G$ that are closer to $a$ than
to $b$. We write $\langle S \rangle$ for the subgraph of $G$ induced by $S \subseteq V(G)$. The following theorem puts forth two fundamental characterizations of partial cubes:
\begin{theorem} \cite{klavzar-book} \label{th:partial-k} For a connected graph $G$, the following statements are equivalent:
\begin{itemize}
\item [(i)] $G$ is a partial cube.
\item [(ii)] $G$ is bipartite, and $\langle W_{ab} \rangle $ and $\langle W_{ba} \rangle$ are convex subgraphs of $G$ for all $ab \in E(G)$.
\item [(iii)] $G$ is bipartite and $\Theta = \Theta^*$.
\end{itemize}
\end{theorem}
Is it also known that if $G$ is a partial cube and $E$ is a $\Theta$-class of $G$, then $G - E$ has exactly two connected components, namely $\langle W_{ab} \rangle $ and $\langle W_{ba} \rangle$, where $ab \in E$. For more information about partial cubes see \cite{klavzar-book}. 

\section{The edge-hyper-Wiener index of partial cubes}

In this section a method for computing the edge-hyper-Wiener index of partial cubes is developed. For this purpose, we need some auxiliary results. We start with the following definition.

\begin{definition}
Let $G$ be a partial cube and let $E_k$ be its $\Theta$-class. Furthermore, let $U$ and $U'$ be the connected components of the graph $G - E_k$. If $e,f \in E(G)$, we define

$$ \delta_k(e,f) = \left\{ \begin{matrix} 1; & e \in E(U) \, \& \, f \in E(U') \ \text{or} \ e \in E(U') \, \& \, f \in E(U),\\ 0;& \text{otherwise.}  \end{matrix} \right. $$

\end{definition}
The following lemma is crucial for our main theorem.
\begin{lemma}
\label{distance}
Let $G$ be a partial cube and let $d$ be the number of its $\Theta$-classes. If $e,f \in E(G)$, then it holds
$$\widehat{d}(e,f) = \sum_{k=1}^{d}\delta_k(e,f).$$
\end{lemma}

\begin{proof}
Let $a$ and $x$ be end-vertices of $e$ and $f$, respectively, such that $d(a,x) = \widehat{d}(e,f)$. Also, let $b$ and $y$ be the remaining end-vertices of $e$ and $f$, respectively. Furthermore, let $P$ be a shortest path between $a$ and $x$. Obviously, $|E(P)| = \widehat{d}(e,f)$. If $E_1, \ldots, E_d$ are $\Theta$-classes of $G$, than for every $k \in \lbrace 1, \ldots, d \rbrace$ we define
$$F_k = E_k \cap E(P).$$ 
Let $k \in \lbrace 1, \ldots, d \rbrace$ and let $U$ and $U'$ be the connected components of the graph $G - E_k$ such that $a \in V(U)$. Now we can show that $|F_k| = \delta_k(e,f)$. We consider the following two cases.
\begin{enumerate}
\item $F_k = \emptyset$ \\
In this case, since no edge of $P$ is in $E_k$, it follows that path $P$ is completely contained in $U$ or $U'$. Hence, $\delta_k(e,f) = 0 =|F_k|$ and the statement is true.
\item $F_k \neq \emptyset $ \\
In this case, since no two edges in a shortest path are in relation $\Theta$, it follows $|F_k|=1$.  Let $uv$ be the edge in $F_k$. Without loss of generality suppose that $V(U) = W_{uv}$ and $V(U')=W_{vu}$. Obviously, $a \in V(U)$ and $x \in V(U')$. We first show that $b \in V(U) = W_{uv}$. If $b \in W_{vu}$, then $d(b,v) < d(b,u) = d(a,u) + 1$ (the equality holds since $G$ is bipartite) and therefore, $d(b,v) \leq d(a,u)$. Hence, $\widehat{d}(e,f) < |E(P)|$, which is a contradiction. In a similar way we can show that $y \in V(U')$. Therefore, $e \in E(U)$ and $f \in E(U')$. It follows that $\delta_k(e,f)= 1 = |F_k|$.
\end{enumerate}
We have proved that $|F_k| = \delta_k(e,f)$ for any $k \in \lbrace 1, \ldots, d \rbrace$. Therefore,
$$\widehat{d}(e,f) = |E(P)| = \sum_{k=1}^{d}|E_k \cap E(P)| = \sum_{k=1}^{d}|F_k| =  \sum_{k=1}^{d}\delta_k(e,f),$$
which completes the proof.
\qed
\end{proof}

To prove the main result, we need to introduce some additional notation. If $G$ is a graph with $\Theta$-classes $E_1, \ldots, E_d$, we denote by $U_k$ and $U_k'$ the connected components of the graph $G - E_k$, where $k \in \lbrace 1, \ldots, d \rbrace$. For any $k,l \in \lbrace 1, \ldots, d \rbrace$ set

\begin{eqnarray*}
M_{kl}^{11} & = & E(U_k) \cap E(U_l), \\
M_{kl}^{10} & = & E(U_k) \cap E(U_l'), \\
M_{kl}^{01} & = & E(U_k') \cap E(U_l), \\
M_{kl}^{00} & = & E(U_k') \cap E(U_l').
\end{eqnarray*}

\noindent
Also, for $k,l \in \lbrace 1, \ldots, d \rbrace$ and $i,j \in \lbrace 0, 1 \rbrace$ we define
$$m_{kl}^{ij} = |M_{kl}^{ij}|.$$

\begin{lemma}
\label{square}
Let $G$ be a partial cube and let $d$ be the number of its $\Theta$-classes. Then
$$\sum_{e \in E(G)} \sum_{f \in E(G)} \widehat{d}(e,f)^2 = 2 \widehat{W}_e(G) + 4 \sum_{k = 1}^{d-1} \sum_{l=k+1}^d \Big( m_{kl}^{11}m_{kl}^{00} + m_{kl}^{10}m_{kl}^{01} \Big).$$
\end{lemma}

\begin{proof}
Using Lemma \ref{distance} we obtain
\begin{eqnarray*}
\sum_{e \in E(G)} \sum_{f \in E(G)} \widehat{d}(e,f)^2 & = & \sum_{e \in E(G)} \sum_{f \in E(G)} \Bigg( \sum_{k=1}^d \delta_k(e,f) \Bigg)^2  \\
& = & \sum_{e \in E(G)} \sum_{f \in E(G)} \sum_{k=1}^d \sum_{l=1}^d \delta_k(e,f) \delta_l(e,f)  \\
& = &    \sum_{k=1}^d \sum_{l=1}^d \Bigg( \sum_{e \in E(G)} \sum_{f \in E(G)} \delta_k(e,f) \delta_l(e,f) \Bigg)  \\
& = &    \sum_{k=1}^d \sum_{e \in E(G)} \sum_{f \in E(G)} \delta_k(e,f) + \sum_{k=1}^d \sum_{\substack{l=1 \\ l \neq k}}^d \Bigg( \sum_{e \in E(G)} \sum_{f \in E(G)} \delta_k(e,f) \delta_l(e,f) \Bigg)  \\
& = &    \sum_{e \in E(G)} \sum_{f \in E(G)} \sum_{k=1}^d \delta_k(e,f) + 2\sum_{k=1}^d \sum_{\substack{l=1 \\ l \neq k}}^d \Bigg(\frac{1}{2} \sum_{e \in E(G)} \sum_{f \in E(G)} \delta_k(e,f) \delta_l(e,f) \Bigg)  \\
& = & \sum_{e \in E(G)} \sum_{f \in E(G)} \widehat{d}(e,f) +  2\sum_{k=1}^d \sum_{\substack{l=1 \\ l \neq k}}^d \Big( m_{kl}^{11}m_{kl}^{00} + m_{kl}^{10}m_{kl}^{01} \Big),
\end{eqnarray*}
where the last equality follows from the obvious fact that $\delta_k(e,f) \delta_l(e,f) = 1$ if and only if $\delta_k(e,f) = 1$ and $ \delta_l(e,f)=1$. Hence,
$$\sum_{e \in E(G)} \sum_{f \in E(G)} \widehat{d}(e,f)^2 = 2 \widehat{W}_e(G) + 4 \sum_{k = 1}^{d-1} \sum_{l=k+1}^d \Big( m_{kl}^{11}m_{kl}^{00} + m_{kl}^{10}m_{kl}^{01} \Big)$$
and we are done.
\qed
\end{proof}

\noindent
Now everything is prepared for the main result of the paper.
\begin{theorem}
\label{glavni}
Let $G$ be a partial cube and let $d$ be the number of its $\Theta$-classes. Then
\begin{equation} \label{glavna} WW_e(G) = 2W_e(G) +  \sum_{k = 1}^{d-1} \sum_{l=k+1}^d \Big( m_{kl}^{11}m_{kl}^{00} + m_{kl}^{10}m_{kl}^{01} \Big) - {{|E(G)|}\choose{2}}.
\end{equation}
\end{theorem}

\begin{proof}
We first notice that
\begin{eqnarray*}
\sum_{e \in E(G)} \sum_{f \in E(G)} d(e,f)^2 & = & \sum_{e \in E(G)} \sum_{\substack{f \in E(G) \\ f \neq e}} d(e,f)^2 \\
& = & \sum_{e \in E(G)} \sum_{\substack{f \in E(G) \\ f \neq e}} \Big( \widehat{d}(e,f) + 1 \Big)^2 \\
& = & \sum_{e \in E(G)} \sum_{\substack{f \in E(G) \\ f \neq e}} \widehat{d}(e,f)^2 + 2\sum_{e \in E(G)} \sum_{\substack{f \in E(G) \\ f \neq e}} \widehat{d}(e,f) + \sum_{e \in E(G)} \sum_{\substack{f \in E(G) \\ f \neq e}} 1.
\end{eqnarray*}
Using Lemma \ref{square} we thus get
$$\sum_{e \in E(G)} \sum_{f \in E(G)} d(e,f)^2 = 2\widehat{W}_e(G) + 4 \sum_{k = 1}^{d-1} \sum_{l=k+1}^d \Big( m_{kl}^{11}m_{kl}^{00} + m_{kl}^{10}m_{kl}^{01} \Big) + 4\widehat{W}_e(G) + |E(G)| \cdot (|E(G)|-1).$$
Therefore, using also Equation \ref{zveza_wie}, we obtain
\begin{eqnarray*}
WW_e(G) & =  & \frac{1}{4} \sum_{e \in E(G)} \sum_{f \in E(G)} d(e,f) + \frac{1}{4} \sum_{e \in E(G)} \sum_{f \in E(G)} d(e,f)^2 \\
& = & \frac{1}{2}W_e(G) + \frac{1}{2}\widehat{W}_e(G) +  \sum_{k = 1}^{d-1} \sum_{l=k+1}^d \Big( m_{kl}^{11}m_{kl}^{00} + m_{kl}^{10}m_{kl}^{01} \Big) \\
& + & \widehat{W}_e(G) + \frac{|E(G)| \cdot (|E(G)|-1)}{4} \\
& = & 2W_e(G) +  \sum_{k = 1}^{d-1} \sum_{l=k+1}^d \Big( m_{kl}^{11}m_{kl}^{00} + m_{kl}^{10}m_{kl}^{01} \Big) - \frac{|E(G)| \cdot (|E(G)|-1)}{2}
\end{eqnarray*}
and the proof is complete.
\qed
\end{proof}
\bigskip

As already mentioned, trees are partial cubes. Moreover, a $\Theta$-class in a tree is just a single edge. Therefore, for a tree $T$ with edges $e_1, \ldots, e_m$, Theorem \ref{glavni} reduces to 
$$WW_e(T) = 2W_e(T) + \sum_{k=1}^{m-1} \sum_{l=k+1}^m m_1(e_k,e_l)m_2(e_k,e_l) - \binom{m}{2},$$ 
where $m_1(e_k,e_l)$ and $m_2(e_k,e_l)$ are the number of edges in the two extremal connected components of the graph $T- \lbrace e_k,e_l \rbrace$, see Figure \ref{trees3}.

\begin{figure}[h!] 
\begin{center}
\includegraphics[scale=0.7]{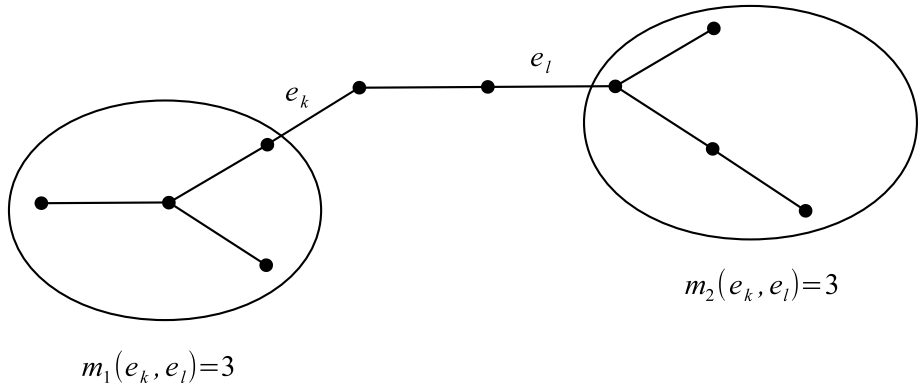}
\end{center}
\caption{\label{trees3} A tree with two extremal components with respect to $e_k$ and $e_l$.}
\end{figure}

\section{Algorithm for benzenoid systems}
\label{algoritem}

Let ${\cal H}$ be the hexagonal (graphite) lattice and let $Z$ be a cycle on it. Then a {\em benzenoid system} is induced by the vertices and edges of ${\cal H}$, lying on $Z$ and in its interior. These graphs are the molecular graphs of the benzenoid hydrocarbons, a large class of organic molecules. For more information about benzenoid systems see \cite{gucy-89}.

\noindent
An \textit{elementary cut} $C$ of a benzenoid system $G$ is a line segment that starts at
the center of a peripheral edge of a benzenoid system $G$,
goes orthogonal to it and ends at the first next peripheral
edge of $G$. By $C$ we sometimes also denote the set of edges that are intersected by the corresponding elementary cut. Elementary cuts in benzenoid systems have been
described and illustrated by numerous examples in several
earlier articles. 

\noindent
The main insight for our consideration
is that every $\Theta$-class of a benzenoid system
$G$ coincide with exactly one of its elementary cuts. Therefore, it is not difficult to check that all benzenoid systems are partial cubes.

In the case of benzenoid systems, Theorem \ref{glavni} provides a particular simple procedure for computing the edge-hyper-Wiener index:
\begin{enumerate}

\item We first compute the edge-Wiener index using the procedure described in \cite{kelenc}, which we briefly repeat: the edge set of a benzenoid system $G$ can be naturally partitioned into sets $E_1, E_2$, and $E_3$ of edges of the same direction. For $i \in \lbrace 1, 2, 3 \rbrace$, set $G_i = G - E_i$. Then the connected components of the graph $G_i$ are paths. The quotient graph $T_i$, $1\le i\le 3$, has these paths as vertices, two such paths (i.e. components of $G_i$) $P'$ and $P''$ being adjacent in $T_i$ if some edge in $E_i$ joins a vertex of $P'$ to a vertex of $P''$. It is known that $T_1$, $T_2$ and $T_3$ are trees. We next extend the quotient trees $T_1$, $T_2$, $T_3$ to weighted trees $(T_i, w_i)$, $(T_i, w'_i)$, $(T_i, w_i, w'_i)$ as follows: 
\begin{itemize}
\item for $C \in V(T_i)$, let $w_i(C)$ be the number of edges in the component $C$ of $G_i$;
\item for $E = C_1C_2 \in E(T_i)$, let $w_i'(E)$ be the number of edges between components $C_1$ and $C_2$.
\end{itemize}

Then the edge-Wiener index of a benzenoid system $G$ can be computed as
\begin{equation} \label{eW} W_e(G)= \sum_{i=1}^3 \left( \widehat{W}_e(T_i, w_i') + W_v(T_i, w_i) + W_{ve}(T_i, w_i, w_i')\right) + \binom{|E(G)|}{2}.
 \end{equation}
To  efficiently compute all the terms in Equation \ref{eW}, some additional notation is needed. If $T$ is a tree and $e \in E(T)$, then the graph $T-e$ consists of two components that will be denoted by  $C_1(e)$ and $C_2(e)$. For a vertex-edge weighted tree $(T,w,w')$ and $e \in E(T)$ set 
$$n_i(e) = \sum_{u \in V(C_i(e))}w_i(u) \qquad {\rm and}\qquad m_i(e) = \sum_{e \in E(C_i(e))}w_i'(e)\,.$$ 
Using this notation we recall the following results (see \cite{kelenc}): 
\begin{eqnarray*}
W_v(T,w)& = & \sum_{e \in E(T)}n_1(e)n_2(e), \\
\widehat{W}_e(T,w') & = & \sum_{e \in E(T)}m_1(e)m_2(e), \\
W_{ve}(T,w,w') & = & \sum_{e \in E(T)} \big(n_1(e)m_2(e) + n_2(e)m_1(e)\big).
\end{eqnarray*}

\item Let us denote the second term in Equation \ref{glavna} by $WW_e^*(G)$, i.e. $$WW_e^*(G) = \sum_{k = 1}^d \sum_{l=k+1}^d \Big( m_{kl}^{11}m_{kl}^{00} + m_{kl}^{10}m_{kl}^{01} \Big).$$ To compute this term we proceed as follows. Let $C_k$ and $C_l$ be two distinct elementary cuts ($\Theta$-classes) of a benzenoid system $G$ such that $l > k$. Then there are two different cases, since the elementary cuts can intersect or not - see Figure \ref{moznosti}. 

\begin{figure}[h!] 
\begin{center}
\includegraphics[scale=0.7]{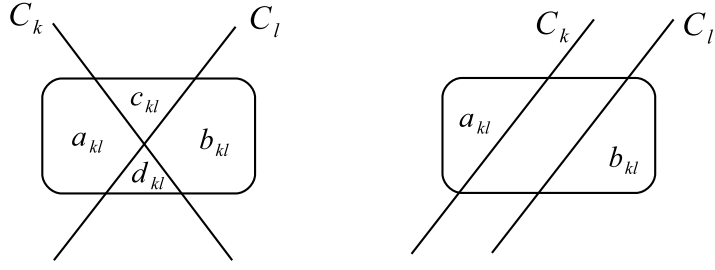}
\end{center}
\caption{\label{moznosti} Two different positions of two elementary cuts.}
\end{figure}

By $a_{kl}$, $b_{kl}$, $c_{kl}$, and $d_{kl}$ we denote the number of edges in the corresponding components of $G-C_k-C_l$. Then the contribution of the pair $C_k, C_l$ to $WW_e^*(G)$ will be denoted by $f(C_k,C_l)$ and we obtain

$$f(C_k,C_l)= 
\begin{cases}
a_{kl}b_{kl} + c_{kl}d_{kl}, & C_k \text{ and } C_l \text{ intersect} \\
a_{kl}b_{kl}, & \text{otherwise}.
\end{cases}
$$

Therefore, term $WW_e^*(G)$ can be computed as
$$WW_e^{*}(G) = \sum_{k = 1}^{d-1} \sum_{l=k+1}^d f(C_k,C_l).$$
\end{enumerate}

\noindent
Finally, we arrive to the algorithm for the computation of the edge-hyper-Wiener index of benzenoid systems. For a given benzenoid system $G$ we first compute its elementary cuts using a procedure called {\tt calculateElementaryCuts} and then the quotient trees $T_i$ using a procedure {\tt calculateQuotientTrees}. For each $T_i$ we first compute the vertex weights $w$ and the edge weights $w'$ using a procedure {\tt calculateWeights}. When the edge-Wiener index is calculated, we sum up all the contributions of pairs of elementary cuts - procedure {\tt calculateContributionsCuts}. Furthermore, Theorem \ref{glavni} is applied to obtain the final result. The algorithm thus reads as follows:
\bigskip

\begin{algorithm}[H]\label{alg:edini}
\SetKwInOut{Input}{Input}\SetKwInOut{Output}{Output}
\DontPrintSemicolon
 
\Input{Benzenoid system $G$ with $m$ edges}
\Output{$WW_e(G)$}

 \SetKwFunction{CQT}{calculateQuotientTrees}
 \SetKwFunction{CET}{calculateElementaryCuts}
 \SetKwFunction{CC}{calculateContributionsCuts}
 \SetKwFunction{CW}{calculateWeights}
 \SetKwFunction{UpdateWeights}{updateWeights}
 $C_1, \ldots, C_d \leftarrow$ \CET($G$)\;
 $(T_1,T_2,T_3) \leftarrow$ \CQT($G$)\;
 \For{$i=1$  {\rm to} $3$}{
 	$(w_i,w'_i) \leftarrow $ \CW($T_i,G$)\;
 	$X_{i,1} \leftarrow W_v(T_i,w_i)$\;
 	$X_{i,2} \leftarrow W_e(T_i,w'_i)$\;
 	$X_{i,3} \leftarrow W_{ve}(T_i,w_i,w'_i)$\;
 	$Y_i \leftarrow X_{i,1}+X_{i,2}+X_{i,3}$
 }
 $W_e(G) \leftarrow Y_1+Y_2+Y_3 + \binom{m}{2}$\;
 $WW_e^*(G) \leftarrow 0$\;
 \For{$k=1$ {\rm to} $d-1$}{
 \For{$l=k+1$ {\rm to} $d$}{
 $f_{k,l} \leftarrow$ \CC($C_k,C_l$)\;
 $WW_e^*(G) \leftarrow$ $WW_e^*(G) + f_{k,l}$\;
 }
 }
 $WW_e(G) \leftarrow 2W_e(G) + WW_e^*(G) - \binom{m}{2}$
 \caption{Edge-Hyper-Wiener Index of Benzenoid Systems}
\end{algorithm} 

\section{Closed formulas for linear polyacenes}

In this section we derive closed formulas for the edge-Wiener index and the edge-hyper-Wiener index of linear polyacenes using the procedure described in Section \ref{algoritem}. If $h \geq 1$, then \textit{linear polyacene} $L_h$ is formed of $h$ linearly connected hexagons, see Figure \ref{L4}. Recall that for $h =1,2,3,4,5$ the graph $L_h$ represents benzene, naphthalene, anthracene, naphthacene, and pentacene (see \cite{gucy-89}).

\begin{figure}[h!] 
\begin{center}
\includegraphics[scale=0.8]{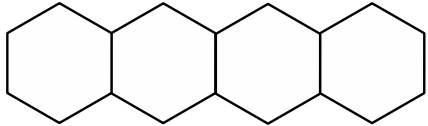}
\end{center}
\caption{\label{L4} Linear polyacene $L_4$.}
\end{figure}

\noindent
Although closed formulas for linear polyacenes were previously obtained in \cite{khormali,edge-hyper}, the results are not correct. For example, it can be computed by hand that $W_e(L_3) = 350$ and $WW_e(L_3) = 812$, but these results do not coincide with the mentioned results.

To obtain closed formulas, we first compute the edge-Wiener index. The corresponding weighted quotient trees for $L_h$ are depicted in Figure \ref{trees2}. 

\begin{figure}[h!] 
\begin{center}
\includegraphics[scale=0.7]{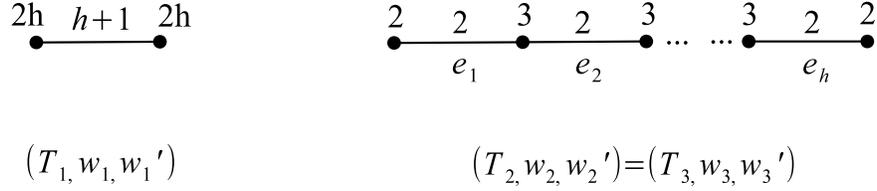}
\end{center}
\caption{\label{trees2} Weighted quotient trees for linear polyacene $L_h$.}
\end{figure}

\noindent
Obviously,

$$W_v(T_1,w_1,w_1') = 4h^2, \quad \widehat{W}_e(T_1,w_1,w_1') = 0, \quad W_{ve}(T_1,w_1,w_1') = 0.$$

\noindent
To compute the corresponding Wiener indices of the second tree, we first notice that for the edge $e_i$, $i \in \lbrace 1, \ldots, h\rbrace$ it holds:

\begin{eqnarray*}
n_1(e_i) & = & 3(i-1) + 2 = 3i-1, \\
n_2(e_i) & = & 3(h-i) + 2 = 3h-3i + 2, \\
m_1(e_i) & = & 2(i-1) = 2i-2, \\
m_2(e_i) & = & 2(h-i) = 2h-2i.
\end{eqnarray*}

\noindent
Therefore, after an elementary calculation we obtain
\begin{eqnarray*}
W_v(T_2,w_2,w_2') = W_v(T_3,w_3,w_3') & = & \sum_{i=1}^{h}(3i-1)(3h-3i+2) = \frac{1}{2}h(3h^2+3h+2), \\
\widehat{W}_e(T_2,w_2,w_2') = \widehat{W}_e(T_3,w_3,w_3') & = & \sum_{i=1}^{h}(2i-2)(2h-2i) = \frac{2}{3}h(h^2-3h+2), \\
W_{ve}(T_2,w_2,w_2') = W_{ve}(T_3,w_3,w_3') & = & \sum_{i=1}^{h}\Big((3i-1)(2h-2i) + (2i-2)(3h-3i+2) \Big) \\
& = & 2(h-1)h^2.
\end{eqnarray*}

\noindent
Also, one can easily see the the number of edges in $L_h$ is
\begin{equation}
\label{pov}
|E(L_h)| = 5h+1.
\end{equation}

\noindent
Hence, Equation \ref{eW} imply that for any $h \geq 1$ it holds

\begin{equation}
\label{edge-wi}
W_e(L_h) = \frac{1}{6}h(50h^2 + 69h + 43).
\end{equation}

Next, we have to compute $WW_e^*(L_h)$. Let $B$, $C_1, C_2, \ldots, C_h$, $D_1, D_2, \ldots, D_h$ be the elementary cuts of $L_h$, where $h \geq 1$. See Figure \ref{cuts}. 

\begin{figure}[h!] 
\begin{center}
\includegraphics[scale=0.8]{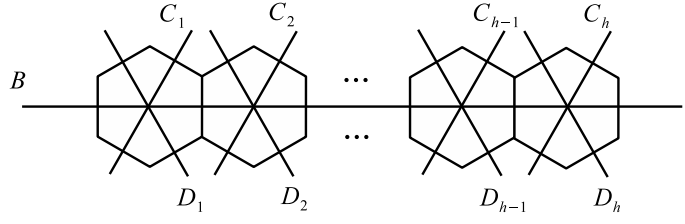}
\end{center}
\caption{\label{cuts} Elementary cuts of $L_h$.}
\end{figure}

\noindent
To obtain the final result, we need to calculate the contributions $f(\cdot,\cdot)$ of all the pairs of elementary cuts. Therefore, the number of edges in specific components (parts) of a graph is shown in Table \ref{tabela}.

\begin{table}[h!]
\centering
\begin{tabular}{|l||c|c|c|c|} 
\hline
Pair of elementary cuts & Part $a$ & Part $b$ & Part $c$ & Part $d$  \\ \hline \hline
$B,C_k$ ($k=1,\ldots,h$) & $2k-1$ & $2h-2k+1$ & $2h-2k$ & $2k-2$  \\ \hline
$B,D_k$ ($k=1,\ldots,h$) & $2k-2$ & $2h-2k$ & $2h-2k+1$ & $2k-1$  \\ \hline
$C_k,C_l$ ($k<l$) & $5k-3$ & $5h-5l+2$ & /& /  \\ \hline
$D_k,D_l$ ($k<l$) & $5k-3$ & $5h-5l+2$ & / & /  \\ \hline
$C_k,D_l$ ($k < l$) & $5k-3$ & $5h-5l+2$ & / & /  \\ \hline
$C_k,D_l$ ($k > l$) & $5l-3$ & $5h-5k+2$ & / & / \\ \hline
$C_k,D_k$ ($k=1,\ldots,h$) & $5k-4$ & $5h-5k+1$ & 0 & 0  \\ \hline
\end{tabular}
\caption{ \label{tabela} Number of edges in the connected components.}
\end{table}

\noindent
After some elementary calculations we obtain
\begin{eqnarray*}
\sum_{k=1}^h\Big( (2k-1)(2h-2k+1)+(2h-2k)(2k-2) \Big) & = & \frac{1}{3}h(4h^2-6h+5), \\
\sum_{k=1}^{h-1} \sum_{l=k+1}^h (5k-3)(5h-5l+2) & = & \frac{1}{24} h(25h^3-70h^2+83h-38), \\
\sum_{k=1}^h (5k-4)(5h-5k+1) & = & \frac{1}{6}h(25h^2-45h+26).
\end{eqnarray*}

\noindent
By summing up the contributions of all pairs of elementary cuts from Table \ref{tabela} we then calculate
\begin{equation}
\label{contri}
WW_e^*(L_h) = \frac{1}{6}h(25h^3 -29h^2 + 14h + 8).
\end{equation}

Finally, using Theorem \ref{glavni}, Equation \ref{edge-wi}, Equation \ref{contri}, and Equation \ref{pov} we deduce that for any $h \geq 1$ it holds
$$WW_e(L_h) = \frac{1}{6}h(25h^3 + 71h^2+77h+79).$$

\bibliography{bib_edge_hyper}

\end{document}